\documentclass[11pt]{article}

\usepackage{amsthm}
\usepackage{amsmath}
\usepackage{amssymb}
\usepackage{amsmath, amsthm, bbm}
\usepackage{framed, fullpage}

\usepackage{mathrsfs}

\usepackage[backref,colorlinks,citecolor=blue,bookmarks=true]{hyperref}

\newtheorem*{theo*}{Theorem}

\newtheorem{theo}{Theorem}

\newtheorem{lemma}{Lemma}[section]
\newtheorem{coro}[theo]{Corollary}
\newtheorem{definition}[lemma]{Definition}
\newtheorem{prop}[lemma]{Proposition}

\newtheorem{claim}[lemma]{Claim}

\newtheorem{remark}[lemma]{Remark}

\makeatletter
\renewenvironment{proof}[1][\proofname]
{\par\pushQED{\qed}
	\normalfont\topsep6\p@\@plus6\p@\relax\trivlist
	\item[\hskip\labelsep\bfseries#1\@addpunct{.}]
	\ignorespaces}
{\popQED\endtrivlist\@endpefalse}
\makeatother

\newcommand{\N}{\mathbb{N}}

\renewcommand{\P}{\mathcal{P}}

\newcommand{\Q}{\mathcal{Q}}
\newcommand{\R}{\mathcal{R}}
\newcommand{\E}{\mathcal{E}}
\newcommand{\A}{\mathcal{A}}

\newcommand{\V}{\mathcal{V}}

\DeclareMathOperator{\poly}{poly}

\DeclareMathOperator{\twr}{T}
\DeclareMathOperator{\wow}{W}

\renewcommand{\a}{\alpha}
\renewcommand{\b}{\beta}
\renewcommand{\d}{\delta}
\newcommand{\g}{\gamma}

\renewcommand{\L}{\mathcal{L}}
\newcommand{\sub}{\subseteq}
\newcommand{\sm}{\setminus}

\newcommand{\e}{\epsilon}

\newcommand{\Z}{\mathcal{Z}}

\newcommand{\K}{\mathcal{K}}
\newcommand{\G}{\mathcal{G}}

\newcommand{\Aside}{\mathbf{A}}
\newcommand{\Bside}{\mathbf{B}}
\newcommand{\Cside}{\mathbf{C}}

\newcommand{\Lside}{\mathbf{L}}
\newcommand{\Rside}{\mathbf{R}}

\newcommand{\Hy}[1]{H}

\DeclareMathOperator{\codeg}{codeg}

\DeclareMathOperator{\Ack}{Ack}

\renewcommand{\k}{\kappa}

\newcommand{\Vside}{\mathbf{V}}

\date{}

\title{A Tight Bound for Hypergraph Regularity I}
\author{Guy Moshkovitz\thanks{School of Mathematical Sciences, Tel Aviv University, Tel Aviv 6997801, Israel. Email: {\tt guymoshkov@gmail.com}. Supported in part by ERC Starting Grant 633509.} 
\and Asaf Shapira\thanks{School of Mathematical Sciences, Tel Aviv University, Tel Aviv 6997801, Israel. Email: {\tt asafico@tauex.tau.ac.il}. Supported in part by ISF Grant 1028/16 and ERC Starting Grant 633509.}}
\begin{document}

\maketitle
\begin{abstract}
The hypergraph regularity lemma -- the extension of Szemer\'edi's graph regularity lemma to the setting of $k$-uniform hypergraphs -- is one of the most celebrated combinatorial results obtained in the past decade.
By now there are several (very different) proofs of this lemma, obtained by Gowers, by Nagle-R\"odl-Schacht-Skokan and by Tao.
Unfortunately, what all these proofs have in common is that they yield regular partitions whose order is given by the $k$-th Ackermann function.
We show that such Ackermann-type bounds are unavoidable for every $k \ge 2$, thus confirming a prediction of Tao.

Prior to our work, the only result of the above type was Gowers' famous lower bound for graph regularity.
In this paper we describe the key new ideas which enable us to overcome several barriers which
stood in the way of establishing such bounds for hypergraphs of higher uniformity.
One of them is a tight bound for a new (very weak) version of the {\em graph} regularity lemma. Using this bound, we prove
a lower bound for any regularity lemma of $3$-uniform hypergraphs that satisfies certain mild conditions.
We then show how to use this result in order to prove a tight bound for
the hypergraph regularity lemmas of Gowers and of Frankl and R\"odl.
We will obtain similar results for hypergraphs of arbitrary uniformity $k\geq 2$ in a subsequent paper.

\end{abstract}

\section{Introduction}

As part of the proof of his eponymous theorem~\cite{Szemeredi75} on arithmetic progressions in dense sets of integers, Szemer\'edi developed (a variant of what is now known as) the graph {\em regularity lemma}~\cite{Szemeredi78}.
The lemma roughly states that the vertex set of every graph can be partitioned into a bounded number of parts such that
almost all the bipartite graphs induced by pairs of parts in the partition are quasi-random.
In the past four decades this lemma has become one of the (if not the) most powerful tools in extremal combinatorics, with applications in many other areas of mathematics. We refer the reader to~\cite{KomlosShSiSz02,RodlSc10}
for more background on the graph regularity lemma, its many variants and its numerous applications.

Perhaps the most important and well-known application of the graph regularity lemma is the original
proof of the {\em triangle removal lemma}, which states that if an $n$-vertex graph $G$ contains only $o(n^3)$ triangles, then one can turn $G$ into a triangle-free graph by removing only $o(n^2)$ edges (see \cite{ConlonFox13} for more details).
It was famously observed by Ruzsa and Szemer\'edi~\cite{RuzsaSz76} that the triangle removal lemma implies Roth's theorem~\cite{Roth54}, the special case of Szemer\'edi's theorem for $3$-term arithmetic progressions.
The problem of extending the triangle removal lemma to the hypergraph
setting was raised by Erd\H{o}s, Frankl and R\"odl~\cite{ErdosFrRo86}. One of the main motivations for obtaining such a result was the observation of Frankl and R\"odl~\cite{FrankRo02} (see also~\cite{Solymosi04}) that such a result would allow one to extend the Ruzsa--Szemer\'edi~\cite{RuzsaSz76} argument and thus obtain an alternative proof of Szemer\'edi's theorem for progressions of arbitrary length.

The quest for a hypergraph regularity lemma, which would allow one to prove a hypergraph removal lemma, took about 20 years.
The first milestone was the result of Frankl and R\"odl~\cite{FrankRo02}, who obtained a regularity lemma for $3$-uniform hypergraphs. About 10 years later, the approach of~\cite{FrankRo02} was extended to hypergraphs of arbitrary uniformity by R\"odl, Skokan, Nagle and Schacht~\cite{NagleRoSc06, RodlSk04}.
At the same time, Gowers~\cite{Gowers07} obtained an alternative version of the regularity lemma for $k$-uniform hypergraphs (from now on we will use $k$-graphs instead of $k$-uniform hypergraphs).
Shortly after, Tao~\cite{Tao06} and R\"odl and Schacht~\cite{RodlSc07,RodlSc07-B} obtained two more versions of the lemma.

As it turned out, the main difficulty with obtaining a regularity lemma for $k$-graphs was defining the correct notion of hypergraph regularity that would:
$(i)$ be strong enough to allow one to prove a counting lemma, and
$(ii)$ be weak enough to be satisfied by every hypergraph (see the discussion in~\cite{Gowers06} for more on this issue).
And indeed, the above-mentioned variants of the hypergraph regularity lemma rely on four different notions of quasi-randomness, which
to this date are still not known to be equivalent\footnote{This should be contrasted with the setting of graphs in which
(almost) all notions of quasi-randomness are not only known to be equivalent but even effectively equivalent. See e.g.~\cite{ChungGrWi89}.} (see~\cite{NaglePoRoSc09} for some partial results). What all of these proofs {\em do} have in common however,
is that they supply only Ackermann-type bounds for the size of a regular partition.\footnote{Another variant of the hypergraph regularity lemma was obtained in~\cite{ElekSz12}. This approach does not supply any quantitative bounds.} 
More precisely, if we let $\Ack_1(x)=2^x$ and then define $\Ack_k(x)$ to be the $x$-times iterated\footnote{$\Ack_2(x)$ is thus a tower of exponents of height $x$, 
	$\Ack_3(x)$ is the so-called wowzer function, etc.} version of $\Ack_{k-1}$, then all the above proofs guarantee to produce a regular partition of a $k$-graph whose order can be bounded from above by an $\Ack_k$-type function.

One of the most important applications of the $k$-graph regularity lemma was that it gave the first explicit
bounds for the multidimensional generalization of Szemer\'edi's theorem, see~\cite{Gowers07}. The original proof of this result, obtained by Furstenberg and Katznelson~\cite{FurstenbergKa78}, relied on Ergodic Theory and thus supplied no quantitative bounds at all.
Examining the reduction between these theorems~\cite{Solymosi04} reveals that if one could improve the Ackermann-type bounds for the $k$-graph regularity
lemma, by obtaining (say) $\Ack_{k_0}$-type upper bounds (for all $k$), then one would obtain the first primitive recursive bounds for the
multidimensional generalization of Szemer\'edi's theorem. Let us note that obtaining such bounds just for
van der Waerden's theorem~\cite{Shelah89} and Szemer\'edi's theorem~\cite{Szemeredi75} (which are two special case) were open problems for many decades till
they were finally solved by Shelah~\cite{Shelah89} and Gowers~\cite{Gowers01}, respectively.
Further applications of the $k$-graph regularity lemma (and the hypergraph removal lemma in particular) are described in~\cite{RodlNaSkScKo05} and~\cite{RodlTeScTo06} as well as in R\"odl's recent ICM survey~\cite{Rodl14}.

A famous result of Gowers~\cite{Gowers97} states that the $\Ack_2$-type upper bounds for graph regularity
are unavoidable. Several improvements~\cite{FoxLo17},
variants~\cite{ConlonFo12,KaliSh13,MoshkovitzSh18} and simplifications~\cite{MoshkovitzSh16} of Gowers' lower bound were recently obtained, but no analogous lower bound was derived even for $3$-graph regularity.
The numerous applications of the hypergraph regularity lemma naturally lead to the question of whether one can improve upon the Ackermann-type
bounds mentioned above and obtain primitive recursive bounds
for the $k$-graph regularity lemma. Tao~\cite{Tao06-h} predicted that the answer to this question is negative, in the sense that one cannot obtain better than $\Ack_k$-type upper bounds for the $k$-graph regularity lemma for every $k \ge 2$.
The main result presented here and in the followup \cite{MSk} confirms this prediction.


\begin{theo}{\bf[Main result, informal statement]}\label{thm:main-informal}
The following holds for every $k\geq 2$: every regularity lemma for $k$-graphs satisfying some
mild conditions can only guarantee to produce partitions of size bounded by an $\Ack_k$-type function.
\end{theo}

In this paper we will focus on proving the key ingredient needed for obtaining Theorem~\ref{thm:main-informal},
stated as Lemma~\ref{theo:core} in Subsection~\ref{subsec:overview}, and on showing how it can be used in order to prove Theorem~\ref{thm:main-informal}
for $k=3$. In a nutshell, the key idea is to use the graph construction given by Lemma~\ref{theo:core} in order
to construct a $3$-graph by taking a certain ``product'' of two graphs that are hard for graph regularity, in order to get
a $3$-graph that is hard for $3$-graph regularity. See the discussion following Lemma~\ref{theo:core} in Subsection~\ref{subsec:overview}.
Dealing with $k=3$ in this paper will allow us to present all the new ideas needed in order to actually prove Theorem~\ref{thm:main-informal}
for arbitrary $k$, in the slightly friendlier setting of $3$-graphs.
In a followup paper~\cite{MSk}, we will show how Lemma~\ref{theo:core} can be used in order to prove
Theorem~\ref{thm:main-informal} for all $k \ge 2$.

In this paper we will also show how to derive from Theorem~\ref{thm:main-informal} tight lower bounds for
the $3$-graph regularity lemmas due to Frankl and R\"odl~\cite{FrankRo02} and to Gowers~\cite{Gowers06}.

\begin{coro}\label{coro:FR-LB}
There is an $\Ack_3$-type lower bound for the $3$-graph regularity lemmas of Frankl and R\"odl~\cite{FrankRo02} and of Gowers~\cite{Gowers06}.
\end{coro}

In \cite{MSk} we will show how to derive from Theorem~\ref{thm:main-informal} a tight lower bound for
the $k$-graph regularity lemma due to R\"odl and Schacht~\cite{RodlSc07}.

\begin{coro}\label{coro:RS-LB}
There is an $\Ack_k$-type lower bound for the $k$-graph regularity lemma of R\"odl and Schacht~\cite{RodlSc07}.
\end{coro}

Before getting into the gory details of the proof, let us informally discuss what we think are some
interesting aspects of the proof of Theorem \ref{thm:main-informal}.

\paragraph{Why is it hard to ``step up''?}
The reason why the upper bound for graph regularity is of tower-type
is that the process of constructing a regular partition of a graph proceeds by a sequence of steps, each increasing the size of the partition exponentially.
The main idea behind Gowers' lower bound for graph regularity~\cite{Gowers97} is in ``reverse engineering'' the proof of the upper bound; in other words,
in showing that (in some sense) the process of building the partition using a sequence of exponential refinements is unavoidable.
Now, a common theme in all proofs of the hypergraph regularity lemma
for $k$-graphs is that they proceed by induction on $k$; that is, in the process of constructing a regular
partition of the input $k$-graph $H$, the proof applies the $(k-1)$-graph regularity lemma on certain $(k-1)$-graphs
derived from $H$. This is why one gets $\Ack_k$-type upper bounds. So with~\cite{Gowers97} in mind, one might guess
that in order to prove a matching lower bound one should ``reverse engineer'' the proof of the upper bound and show that such a process is unavoidable. However, this turns out to be false! As we argued in~\cite{MoshkovitzSh18}, in order to prove an {\em upper bound} for (say) $3$-graph regularity it is in fact enough to iterate a relaxed version of graph regularity which we call the ``sparse regular approximation lemma'' (SRAL for short).
Therefore, in order to prove an $\Ack_3$-type {\em lower bound} for $3$-graph
regularity one cannot simply ``step up'' an $\Ack_2$-type lower bound for graph regularity. Indeed, a necessary condition
would be to prove an $\Ack_2$-type lower bound for SRAL. See also the discussion following Lemma~\ref{theo:core} in Subsection \ref{subsec:overview}
on how do we actually use a graph construction in order to get a $3$-graph construction.

\paragraph{A new notion of graph/hypergraph regularity:}
In a recent paper \cite{MoshkovitzSh18} we proved an $\Ack_2$-type
lower bound for SRAL.
As it turned out, even this lower bound was not enough to allow us to step up the graph lower bound
into a $3$-graph lower bound. To remedy this, in the present paper we introduce an even weaker notion of graph/hypergraph regularity
which we call $\langle \d \rangle$-regularity. This notion seems to be right at the correct level of ``strength'';
on the one hand, it is strong enough to allow one to prove $\Ack_{k-1}$-type lower bounds for $(k-1)$-graph regularity, while at the
same time weak enough to allow one to induct, that is, to use it in order to then prove $\Ack_{k}$-type lower bounds for $k$-graph regularity.
Another critical feature of our new notion of hypergraph regularity is that it has (almost) nothing to do with hypergraphs!
A disconcerting aspect of all proofs of the hypergraph regularity lemma is that they involve a very complicated nested/inductive structure.
Furthermore, one has to introduce an elaborate hierarchy of constants that controls how regular one level of the partition is compared to
the previous one. What is thus nice about our new notion is that it involves only various kinds of instances of graph $\langle \d \rangle$-regularity.
As a result, our proof is (relatively!) simple.

\paragraph{How do we find witnesses for $3$-graph irregularity?}
The key idea in Gowers' lower bound~\cite{Gowers97} for graph regularity was in constructing a graph $G$, based on a sequence of
partitions ${\cal P}_1,{\cal P}_2,\ldots$ of $V(G)$, with the following
inductive property: if a vertex partition $\Z$ refines ${\cal P}_i$
but does not refine ${\cal P}_{i+1}$ then $\Z$ is not $\epsilon$-regular.
The key step of the proof of~\cite{Gowers97} is in finding witnesses showing
that pairs of clusters of $\Z$ are irregular. The main difficulty in extending this
strategy to $k$-graphs already reveals itself in the setting of $3$-graphs. In a nutshell, while
in graphs, a witness to irregularity of a pair of clusters $A,B \in \Z$ is
{\em any} pair of large subsets $A' \sub A$ and $B' \sub B$, in the setting of $3$-graphs
we have to find three large edge-sets (called a {\em triad}, see Section~\ref{sec:FR}) that
have an additional property: they must together form a graph containing many triangles.
It thus seems quite hard to extend Gowers' approach already to the setting of $3$-graphs. By working
with the much weaker notion of $\langle \d \rangle$-regularity, we circumvent this
issue since two of the edges sets in our version of a triad are always complete bipartite graphs.
See Subsection~\ref{subsec:definitions}.

\paragraph{What is then the meaning of Theorem \ref{thm:main-informal}?}
Our main result, stated formally as Theorem~\ref{theo:main}, establishes an $\Ack_3$-type lower bound
for $\langle \d \rangle$-regularity of $3$-graphs, that is, for a specific new
version of the hypergraph regularity lemma. Therefore, we immediately get $\Ack_3$-type lower bounds for
any $3$-graph regularity lemma which is at least as strong as our new lemma, that is, for any lemma
whose requirements/guarantees imply those that are needed in order to satisfy our new notion of regularity.
In particular, we will prove Corollary \ref{coro:FR-LB} by showing that the regularity notions
used in these lemmas are at least as strong as $\langle \d \rangle$-regularity.

In \cite{MSk} we will prove Theorem~\ref{thm:main-informal} in its full generality by extending Theorem~\ref{theo:main}
to arbitrary $k$-graphs. This proof, though technically more involved, will be
quite similar at its core to the way we derive Theorem~\ref{theo:main} from Lemma~\ref{theo:core} in the present paper.
The deduction of Corollary \ref{coro:RS-LB}, which appears in \cite{MSk}, will also turn out to be quite similar to the way
Corollary \ref{coro:FR-LB} is derived from Theorem~\ref{theo:main} in the present paper

\paragraph{How strong is our lower bound?} Since Theorem \ref{thm:main-informal} gives a lower bound for
$\langle \d \rangle$-regularity and Corollaries \ref{coro:FR-LB} and \ref{coro:RS-LB} show that
this notion is at least as weak as previously used notions of regularity, it is natural to ask:
$(i)$ is this notion equivalent to one of the other notions? $(ii)$ is this notion strong enough for proving the
hypergraph removal lemma, which was one of the main reasons for developing
the hypergraph regularity lemma? We will prove that the answer to both questions is {\em negative} by showing that already for
graphs, $\langle \d \rangle$-regularity (for $\d$ a fixed constant) is not strong enough even for proving the triangle removal lemma.
This of course makes our lower bound even stronger as it already applies to a very weak notion of regularity.
In a nutshell, the proof proceeds by first taking a random tripartite graph, showing (using routine probabilistic
arguments) that with high probability the graph is $\langle \d \rangle$-regular yet contains a small number of triangles.
One then shows that removing these triangles, and then taking a blowup of the resulting graph, gives a triangle-free graph of positive density that is $\langle \d \rangle$-regular. The full details will appear in \cite{MSk}.

\paragraph{How tight is our bound?} Roughly speaking, we will show that for a $k$-graph with $pn^k$ edges,
every $\langle \d \rangle$-regular partition has order at least $\Ack_k(\log 1/p)$. In a recent paper \cite{MoshkovitzSh16}
we proved that in graphs, one can prove a matching $\Ack_2(\log 1/p)$ upper bound, even for a slightly stronger notion than $\langle \d \rangle$-regularity.
This allowed us to obtain a new proof of Fox's $\Ack_2(\log 1/\epsilon)$ upper bound for the graph removal lemma \cite{Fox11} (since the stronger notion allows to count small subgraphs). We believe that it should
be possible to match our lower bounds with $\Ack_k(\log 1/p)$ upper bounds (even for a slightly stronger notion analogous to the one used in \cite{MoshkovitzSh16}). We think that it should be possible to deduce from such an upper bound an $\Ack_k(\log 1/\epsilon)$ upper bound
for the $k$-graph removal lemma. The best known bounds for this problem are (at least) $\Ack_k(\poly(1/\epsilon))$.


\subsection{Paper overview}

In Section~\ref{sec:define} we will first define the new notion of hypergraph regularity, which we term $\langle \d \rangle$-regularity,
for which we will prove our main lower bound. We will then give the formal version of Theorem~\ref{thm:main-informal} (see Theorem~\ref{theo:main}).
This will be followed by the statement of our core technical result, Lemma~\ref{theo:core},
and an overview of how this technical result is used in the proof of Theorem~\ref{theo:main}.
The proof of Theorem~\ref{theo:main} appears in Section~\ref{sec:LB}. 
We refer the reader to~\cite{MSk} for the proof of Lemma~\ref{theo:core}.
In Section \ref{sec:FR} we prove Corollary~\ref{coro:FR-LB}.
In Appendix~\ref{sec:FR-appendix} we give the proof of certain technical claims missing from Section~\ref{sec:FR}.


\section{$\langle \d \rangle$-regularity and Proof Overview}\label{sec:define}

Formally, a \emph{$3$-graph} is a pair $H=(V,E)$, where $V=V(H)$ is the vertex set and $E=E(H) \sub \binom{V}{3}$ is the edge set of $H$.
The number of edges of $H$ is denoted $e(H)$ (i.e., $e(H)=|E|$).
The $3$-graph $H$ is \emph{$3$-partite} on (disjoint) vertex classes $(V_1,V_2,V_3)$ if every edge of $H$ has a vertex from each $V_i$.
The \emph{density} of a $3$-partite $3$-graph $H$ is $e(H)/\prod_{i=1}^3 |V_i|$.
For a bipartite graph $G$, the set of edges of $G$ between disjoint vertex subsets $A$ and $B$ is denoted by $E_G(A,B)$; the density of $G$ between $A$ and $B$ is denoted by $d_G(A,B)=e_G(A,B)/|A||B|$, where $e_G(A,B)=|E_G(A,B)|$. We use $d(A,B)$ if $G$ is clear from context.
When it is clear from context, we sometimes identify a hypergraph with its edge set. In particular, we will write $V_1 \times V_2$ for the complete bipartite graph on vertex classes $(V_1,V_2)$.
%
%
For partitions $\P,\Q$ of the same underlying set, we say that $\Q$ \emph{refines} $\P$, denoted $\Q \prec \P$, if every member of $\Q$ is contained in a member of $\P$.
We say that $\P$ is \emph{equitable} if all its members have the same size.\footnote{In a regularity lemma one allows the parts to differ in size by at most $1$ so that it applies to all (hyper-)graphs. For our lower bound this is unnecessary.}
We use the notation $x \pm \e$ for a number lying in the interval $[x-\e,\,x+\e]$.

In the following definition, and in the rest of the paper, we will sometimes identify a graph or a $3$-graph with its edge set when the vertex set is clear from context.

\begin{definition}[$2$-partition]\label{def:2-partition}
A \emph{$2$-partition} $(\Z,\E)$ on a vertex set $V$ consists of a partition $\Z$ of $V$ and a family of edge disjoint bipartite graphs
$\E$ so that:
\begin{itemize}
\item Every $E \in \E$ is a bipartite graph 
whose two vertex sets are distinct $Z,Z' \in \Z$.
\item For every $Z \neq Z' \in \Z$, the complete bipartite graph $Z \times Z'$ is the union of graphs from $\E$.
\end{itemize}
\end{definition}

Put differently, a $2$-partition consists of vertex partition $\Z$ and a collection of bipartite graphs $\E$ such that
$\E$ is a refinement of the collection of complete bipartite graphs $\{Z \times Z' : Z \neq Z' \in \Z \}$.

\subsection{$\langle \d \rangle$-regularity of graphs and hypergraphs}\label{subsec:definitions}

In this subsection we define our new\footnote{For $k=3$, related notions of regularity were studied in~\cite{ReiherRoSc16,Towsner17}.} notion of $\langle\d\rangle$-regularity, first for graphs and then for $3$-graphs in Definition~\ref{def:k-reg} below.
Let us first recall Szemer\'edi's notion of $\epsilon$-regularity.
A bipartite graph on $(A,B)$ is \emph{$\e$-regular} if for all subsets $A' \sub A$, $B' \sub B$ with $|A'|\ge\e|A|$, $|B'|\ge\e|B|$ we have $|d(A',B') -d(A,B)| \le \e$.
A vertex partition $\P$ of a graph is $\e$-regular
if
the bipartite graph induced on each but at most $\e|\P|^2$ of the pairs $(A,B)$ with $A \neq B \in \P$ is $\e$-regular.
Szemer\'edi's graph regularity lemma says that every graph
has an $\e$-regular equipartition of order at most some $\Ack_2(\poly(1/\e))$.
We now introduce a weaker notion of graph regularity which we will use throughout the paper.
\begin{definition}[graph $\langle\d\rangle$-regularity]\label{def:star-regular}
	A bipartite graph $G$ on $(A,B)$ is \emph{$\langle \d \rangle$-regular} if for all subsets $A' \sub A$, $B' \sub B$ with $|A'| \ge \d|A|$, $|B'|\ge\d|B|$ we have $d_G(A',B') \ge \frac12 d_G(A,B)$.\\
	A vertex partition $\P$ of a graph $G$ is \emph{$\langle \d \rangle$-regular}
	if one can add/remove at most $\d \cdot e(G)$ edges so that the bipartite graph induced on each $(A,B)$ with $A \neq B \in \P$ is $\langle \d \rangle$-regular.
	%
\end{definition}

For the reader worried that in Definition~\ref{def:star-regular} we merely replaced the $\e$
from the definition of $\e$-regularity with $\d$, we refer to the discussion following Theorem~\ref{theo:main} below.


%
%

The definition of $\langle\d\rangle$-regularity for hypergraphs involves the $\langle\d\rangle$-regularity notion for graphs, applied to certain auxiliary graphs which are defined as follows.


\begin{definition}[The auxiliary graph $G_{H}^i$]\label{def:aux}
For a $3$-partite $3$-graph $H$ on vertex classes $(V_1,V_2,V_3)$,
we define a bipartite graph $G_{H}^1$ on the vertex classes $(V_2 \times V_3,\,V_1)$ by
$$E(G_{H}^1) = \big\{ ((v_2,v_3),v_1) \,\big\vert\, (v_1,v_2,v_3) \in E(H) \big\} \;.$$
The graphs $G_{H}^2$ and $G_{H}^3$ are defined in an analogous manner.
\end{definition}

Importantly, for a $2$-partition (as defined in Definition~\ref{def:2-partition}) to be $\langle\d\rangle$-regular it must first satisfy a requirement on the regularity of its parts.


\begin{definition}[$\langle\d\rangle$-good partition]\label{def:k-good}
A $2$-partition $(\Z,\E)$ on $V$ is \emph{$\langle\d\rangle$-good} if all bipartite graphs in $\E$
(between any two distinct vertex clusters of $\Z$) are $\langle \d \rangle$-regular.
\end{definition}




For a $2$-partition $(\Z,\E)$ of a $3$-partite $3$-graph on vertex classes $(V_1,V_2,V_3)$ with $\Z \prec \{V_1,V_2,V_3\}$, for every $1 \le i \le 3$ we denote $\Z_i = \{Z \in \Z \,\vert\, Z \sub V_i\}$, and we denote $\E_i = \{E \in \E \,\vert\, E \sub V_j \times V_k\}$ where $\{i,j,k\}=\{1,2,3\}$.
So for example, $\E_1$ is thus a partition of $V_2 \times V_3$.

\begin{definition}[$\langle\d\rangle$-regular partition]\label{def:k-reg}
Let $H$ be a $3$-partite $3$-graph on vertex classes $(V_1,V_2,V_3)$
and $(\Z,\E)$ be a $\langle \d \rangle$-good $2$-partition with  $\Z \prec \{V_1,V_2,V_3\}$.
We say that $(\Z,\E)$ is a \emph{$\langle \d \rangle$-regular} partition of $H$ if
for every $1 \le i \le 3$,
$\E_i \cup \Z_i$ is a $\langle \d \rangle$-regular partition of $G_H^i$.
%
%
%
	%
	%
\end{definition}

\subsection{Formal statement of the main result}

We are now ready to formally state our tight lower bound for $3$-graph $\langle \d \rangle$-regularity (the formal version of
Theorem~\ref{thm:main-informal} above for $k=3$). Recall that we define the {\em tower} functions $\twr(x)$ to be a tower of exponents of height $x$,
and then define the {\em wowzer} function $\wow(x)$ to be the $x$-times iterated tower function, that is
$\wow(x)= \underbrace{\twr(\twr(\cdots(\twr(1))\cdots))}_{x \text{ times}}$.

\begin{theo}[Main result]\label{theo:main}
For every $s \in \N$ there is a $3$-partite $3$-graph $H$ on vertex classes of equal size and of density at least $2^{-s}$,
and a partition $\V_0$ of $V(H)$ with $|\V_0| \le 2^{300}$, such that if
$(\Z,\E)$ is a $\langle 2^{-73} \rangle$-regular partition of $H$ with
$\Z \prec \V_0$ then $|\Z| \ge \wow(s)$.

%

\end{theo}

Let us draw the reader's attention to an important and perhaps surprising aspect of Theorem~\ref{theo:main}.
All the known tower-type lower bounds for graph regularity depend on the error parameter $\epsilon$,
that is, they show the existence of graphs $G$ with the property that every $\epsilon$-regular partition of $G$ is of order at least $\Ack_2(\poly(1/\e))$.
This should be contrasted with the fact that our lower bounds for $\langle \d \rangle$-regularity holds for a {\em fixed} error parameter $\delta$.
Indeed, instead of the dependence on the error parameter, our lower bound depends on the {\em density} of the graph.
This delicate difference makes it possible for us to prove Theorem~\ref{theo:main} by iterating the construction described in the next subsection.

\subsection{The core construction and proof overview}\label{subsec:overview}

The graph construction in Lemma~\ref{theo:core} below is the main technical result we will need in order to prove Theorem~\ref{theo:main}.
We will first need to define ``approximate'' refinement (a notion that goes back to Gowers~\cite{Gowers97}).
\begin{definition}[Approximate refinements]
For sets $S,T$ we write $S \sub_\b T$ if $|S \sm T| < \b|S|$.
For a partition $\P$ we write $S \in_\b \P$ if $S \sub_\b P$ for some $P \in \P$.
For partitions $\P,\Q$ of the same set of size $n$ we write $\Q \prec_\b \P$ if
$$\sum_{\substack{Q \in \Q\colon\\Q \notin_\b \P}} |Q| \le \b n \;.$$
\end{definition}
Note that for $\Q$ equitable, $\Q \prec_\b \P$ if and only if
all but at most $\b|\Q|$ parts $Q \in \Q$ satisfy $Q \in_\b \P$.
We note that throughout the paper we will only use approximate refinements with $\b \le 1/2$, and so if $S \in_\b \P$ then $S \sub_\b P$ for a unique $P \in \P$.

We stress that in Lemma~\ref{theo:core} below we only use notions related to graphs. In particular, $\langle \d \rangle$-regularity refers to Definition~\ref{def:star-regular}.

\begin{lemma}[Core construction]\label{theo:core}
Let $\Lside$ and $\Rside$ be disjoint sets. Let
$\L_1 \succ \cdots \succ \L_s$ and $\R_1 \succ \cdots \succ \R_s$ be two sequences of $s$ successively refined equipartitions of $\Lside$ and $\Rside$, respectively,
that satisfy for  every $i \ge 1$ that:
\begin{enumerate}
\item\label{item:core-minR}
$|\R_i|$ is a power of $2$ and $|\R_1| \ge 2^{200}$,
\item\label{item:core-expR} $|\R_{i+1}| \ge 4|\R_i|$  if $i < s$,
\item\label{item:core-expL} $|\L_i| = 2^{|\R_i|/2^{i+10}}$.
\end{enumerate}
Then there exists a sequence of $s$ successively refined edge equipartitions $\G_1 \succ \cdots \succ \G_s$ of $\Lside \times \Rside$ such that for every $1 \le j \le s$, $|\G_j|=2^j$, and the following holds for every $G \in \G_j$ and $\d \le 2^{-20}$.
For every $\langle \d \rangle$-regular partition $\P \cup \Q$ of $G$, where $\P$ and $\Q$ are partitions of $\Lside$ and $\Rside$, respectively, and every $1 \le i \le j$, if $\Q \prec_{2^{-9}} \R_{i}$ then $\P \prec_{\g} \L_{i}$ with $\g = \max\{2^{5}\sqrt{\d},\, 32/\sqrt[6]{|\R_1|} \}$.
\end{lemma}
\begin{remark}\label{remequi}
Every $G \in \G_j$ is a bipartite graph of density $2^{-j}$ since $\G_j$ is equitable.
\end{remark}

As mentioned before, the proof of Lemma~\ref{theo:core} appears in~\cite{MSk}.
Let us end this section by explaining the role Lemma~\ref{theo:core} plays in the proof of Theorem \ref{theo:main}.

\paragraph{Using graphs to construct $3$-graphs:}
Perhaps the most surprising aspect of the proof of Theorem~\ref{theo:main} is that in order to construct a $3$-graph we also use the graph construction of Lemma~\ref{theo:core} in a somewhat unexpected way.
In this case, $\Lside$ will be a complete bipartite graph and the $\L_i$'s will be partitions of this complete bipartite graph themselves given by another application of Lemma~\ref{theo:core}.
The partitions will be of wowzer-type growth, and the second application of Lemma~\ref{theo:core} will ``multiply'' the graph partitions (given by the $\L_i$'s) to
give a partition of the complete $3$-partite $3$-graph into $3$-graphs that are hard for $\langle \d \rangle$-regularity.
We will take $H$ in Theorem~\ref{theo:main} to be an arbitrary $3$-graph in this partition.


\paragraph{Why is Lemma~\ref{theo:core} one-sided?}
As is evident from the statement of Lemma~\ref{theo:core}, it is one-sided in nature; that is, under the premise that the partition $\Q$
refines $\R_i$ we may conclude that $\P$ refines $\L_i$.
It is natural to ask if one can do away with this assumption, that is, be able to show that
under the same assumptions $\Q$ refines $\R_i$ and $\P$ refines $\L_i$.
As we mentioned in the previous item, in order to prove a wowzer-type lower bound for $3$-graph
regularity we have to apply Lemma~\ref{theo:core} with a sequence of partitions that grows as a wowzer-type function.
Now, in this setting, Lemma~\ref{theo:core} does not hold without the one-sided assumption, because if it did, then
one would have been able to prove a wowzer-type lower bound for graph $\langle \d \rangle$-regularity, and hence also for Szemer\'edi's regularity lemma.
Put differently, if one wishes to have a construction that holds with arbitrarily fast growing partition sizes, then
one has to introduce the one-sided assumption.

\paragraph{How do we remove the one-sided assumption?}
The proof of Theorem \ref{theo:main} proceeds by first proving a one-sided version of Theorem \ref{theo:main},
stated as Lemma~\ref{lemma:ind-k}. In order to get a construction that does not require such a one-sided assumption,
we will need one final trick; we will take $6$ clusters of
vertices and arrange $6$ copies
of this one-sided construction
along the $3$-edges of a cycle. This will give us a ``circle of implications'' that will eliminate the one-sided assumption.
See Subsection \ref{subsec:pasting}.

\section{Proof of Theorem~\ref{theo:main}}\label{sec:LB}

\renewcommand{\k}{r}

\renewcommand{\t}{t}
\newcommand{\w}{w}

\newcommand{\GG}{\mathbf{G}}
\newcommand{\FF}{\mathbf{F}}
\newcommand{\VV}{\mathbf{V}}

\renewcommand{\Hy}[1]{H_{{#1}}}
\renewcommand{\A}{A}

\newcommand{\subs}{\subset_*}
\newcommand{\pad}{P}
\renewcommand{\K}{\mathcal{K}}
\newcommand{\U}{U}

\renewcommand{\k}{k}

\renewcommand{\K}{\mathcal{K}}

\renewcommand{\r}{k}

The purpose of this section is to prove the main result, Theorem~\ref{theo:main}.
This section is self-contained save for the application of Lemma~\ref{theo:core}.
The key step of the proof, stated as Lemma~\ref{lemma:ind-k} and proved in Subsection \ref{subsec:key}, relies on a subtle construction that uses Lemma~\ref{theo:core} twice. This lemma only gives a ``one-sided'' lower bound for $3$-graph regularity, in the spirit of Lemma~\ref{theo:core}.
In Subsection~\ref{subsec:pasting} we show how to use Lemma~\ref{lemma:ind-k} in order to complete the proof of Theorem~\ref{theo:main}.

We first observe a simple yet crucial property of $2$-partitions, stated as Claim~\ref{claim:uniform-refinement} below, which we will need later.
This property relates $\d$-refinements of partitions and  $\langle \d \rangle$-regularity of partitions,
and relies
on Claim~\ref{claim:refinement-union}. Here, as well as in the rest of this section, we will use
the definitions and notations introduced in Section \ref{sec:define}.
In particular, recall that if a vertex partition $\Z$ of vertex classes $(V_1,V_2,V_3)$ satisfies $\Z \prec \{V_1,V_2,V_3\}$,
then for every $1 \le i \le 3$ we denote $\Z_i = \{Z \in \Z \,\vert\, Z \sub V_i\}$.
Moreover, if a $2$-partition $(\Z,\E)$, satisfies $\Z \prec \{V_1,V_2,V_3\}$ we denote $\E_i = \{E \in \E \,\vert\, E \sub V_j \times V_k\}$ where $\{i,j,k\}=\{1,2,3\}$. We will first need the following easy claim regarding the union of $\langle \d\rangle$-regular graphs.

\begin{claim}\label{claim:star-union}
	Let $G_1,\ldots,G_\ell$ be mutually edge-disjoint bipartite graphs on the same vertex classes $(Z,Z')$.
	If every $G_i$ is $\langle \d \rangle$-regular then $G=\bigcup_{i=1}^\ell G_i$ is also $\langle \d \rangle$-regular.
\end{claim}
\begin{proof}
	Let $S \sub Z$, $S' \sub Z'$ with $|S| \ge \d|Z|$, $|S'| \ge \d|Z'|$.
	Then
	$$d_G(S,S') = \frac{e_G(S,S')}{|S||S'|}
	= \sum_{i=1}^\ell \frac{e_{G_i}(S,S')}{|S||S'|}
	= \sum_{i=1}^\ell d_{G_i}(S,S')
	\ge \sum_{i=1}^\ell \frac12 d_{G_i}(Z,Z')
	= \frac12 d_{G}(Z,Z') \;,$$
	where the second and last equalities follow from the mutual disjointness of the $G_i$, and the inequality follows from the $\langle \d \rangle$-regularity of each $G_i$.
	Thus, $G$ is $\langle \d \rangle$-regular, as claimed.
\end{proof}

We use the following claim regarding approximate refinements.

\begin{claim}\label{claim:refinement-union}
	If $\Q \prec_\d \P$ then there exist $P \in \P$ and $Q$ that is a union of members of $\Q$ such that $|P \triangle Q| \le 3\d|P|$.
\end{claim}
\begin{proof}
	For each $P\in \P$ let $\Q(P) = \{Q \in \Q \colon Q \sub_\d P\}$,
	and denote $P_\Q = \bigcup_{Q \in \Q(P)} Q$.
	We have
	\begin{align*}
	\sum_{P \in \P} |P \triangle P_\Q|
	&= \sum_{P \in \P} |P_\Q \sm P| + \sum_{P \in \P} |P \sm P_\Q|
	= \sum_{P \in \P} \sum_{\substack{Q \in \Q \colon\\Q \sub_\d P}} |Q \sm P|
	+ \sum_{P \in \P} \sum_{\substack{Q \in \Q \colon\\Q \nsubseteq_\d P}} |Q \cap P| \\
	&\le \sum_{P \in \P} \sum_{\substack{Q \in \Q \colon\\Q \sub_\d P}} \d|Q|
	+ \Big( \sum_{\substack{Q \in \Q\colon\\Q \notin_\d \P}} |Q|
	+ \sum_{\substack{Q \in \Q \colon\\Q \in_\d \P}} \d|Q| \Big)
	\le 3\d\sum_{Q \in \Q} |Q|
	= 3\d\sum_{P \in \P} |P| \;,
	\end{align*}
	where the last inequality uses the statement's assumption $\Q \prec_\d \P$ to bound the middle summand.
	By averaging, there exists $P \in \P$ such that $|P \triangle P_\Q| \le 3\d|P|$, thus completing the proof.
\end{proof}

The property of $2$-partitions that we need is as follows.
\begin{claim}\label{claim:uniform-refinement}
	Let $\P=(\Z,\E)$ be a $2$-partition with  $\Z \prec \{V_1,V_2,V_3\}$,
	and let $\G$ be a partition of $V_1\times V_2$ with $\E_3 \prec_\d \G$.
	If $(\Z,\E)$ is $\langle \d \rangle$-good
	then $\Z_1 \cup \Z_2$ is a $\langle 3\d \rangle$-regular partition of some $G \in \G$.
\end{claim}

\begin{proof}
	Put $\E=\E_3$.
	By Claim~\ref{claim:refinement-union}, since $\E \prec_\d \G$ there exist $G \in \G$ (a bipartite graph on $(V_1,V_2)$) and $G_\E$ that is a union of members of $\E$ (and thus also a bipartite graph on $(V_1,V_2)$) such that $|G \triangle G_\E| \le 3\d|G|$.
	Letting $Z_1 \in \Z_1$, $Z_2 \in \Z_2$, to complete the proof it suffices to show that the induced bipartite graph $G_\E[Z_1,Z_2]$ is $\langle \d \rangle$-regular (recall Definition~\ref{def:star-regular}).
	By Definition~\ref{def:2-partition}, $G_\E[Z_1,Z_2]$ is a union of bipartite graphs from $\E$ on $(Z_1,Z_2)$.
	Since every graph in $\E$ is $\langle \d \rangle$-regular by the statement's assumption that $(\Z,\E)$ is $\langle \d \rangle$-good (recall Definition~\ref{def:k-good}), we have that $G_\E[Z_1,Z_2]$ is a union of $\langle \d \rangle$-regular bipartite graphs on $(Z_1,Z_2)$.
	By Claim~\ref{claim:star-union}, $G_\E[Z_1,Z_2]$ is $\langle \d \rangle$-regular as well, thus completing the proof.
\end{proof}

We will later need the following easy (but slightly tedious to state) claim.
\begin{claim}\label{claim:restriction}
	Let $H$ be a $3$-partite $3$-graph on vertex classes $(V_1,V_2,V_3)$, and let $H'$ be the induced $3$-partite $3$-graph on vertex classes $(V_1',V_2',V_3')$ with $V_i' \sub V_i$ and $\a \cdot e(H)$ edges.
	If $(\Z,\E)$ is a $\langle \d \rangle$-regular partition of $H$ with
	$\Z \prec \bigcup_{i=1}^3 \{V_i,\,V_i \sm V_i'\}$
	then its restriction $(\Z',\E')$ to $V(H')$ is a $\langle \d/\a \rangle$-regular partition of $H'$.
\end{claim}
\begin{proof}
	Recall Definition~\ref{def:k-reg}.
	Clearly, $(\Z',\E')$ is $\langle \d \rangle$-good. We will show that $\E'_1 \cup \Z'_1$ is a $\langle \d/\a \rangle$-regular partition of $G^1_{H'}$.
	The argument for $G^2_{H'}$, $G^3_{H'}$ will be analogous, hence the proof would follow.
	Observe that $G^1_{H'}$ is an induced subgraph of $G^1_{H}$, namely, $G^1_{H'} = G^i_{H}[V_2' \times V_3',\, V_1']$.
	By assumption, $e(H') = \a e(H)$, and thus $e(G^1_{H'}) = \a e(G^1_{H})$.
	By the statement's assumption on $\Z$ and since $\E_1 \cup \Z_1$ is a $\langle \d \rangle$-regular partition of $G^1_{H}$, we deduce---by adding/removing at most $\d e(G^1_{H}) =  (\d/\a)e(G^1_{H'})$ edges of $G^1_{H'}$---that $\E'_1 \cup \Z'_1$ is a $\langle \d/\a \rangle$-regular partition of $G_{H'}^1$.
	As explained above, this completes the proof.
\end{proof}

Finally, we will need the following claim regarding approximate refinements.
\begin{claim}\label{claim:refinement-size}
	If $\Q \prec_{1/2} \P$ and $\P$ is equitable then $|\Q| \ge \frac14|\P|$.
\end{claim}
\begin{proof}
	We claim that the underlying set $U$ has a subset $U^*$ of size $|U^*|\ge \frac14|U|$ such that the partitions $\Q^*=\{Q \cap U^* \,\vert\, Q \in \Q \} \setminus \{\emptyset\}$ and $\P^*=\{P \cap U^* \,\vert\, P \in \P \} \setminus \{\emptyset\}$
	of $U^*$ satisfy $\Q^* \prec \P^*$.
	Indeed, let $U^* = \bigcup_{Q} Q \cap P_Q$ where the union is over all $Q \in \Q$ satisfying $Q \sub_{1/2} P_Q$ for a (unique) $P_Q \in \P$.
	As claimed, $|U^*| = \sum_{Q \in_{1/2} \P} |Q \cap P_Q| \ge \sum_{Q \in_{1/2} \P} \frac12|Q| \ge \frac14|U|$, using $\Q \prec_{1/2} \P$ for the last inequality.
	Now, since $\P$ is equitable, $|\P^*| \ge \frac14|\P|$.
	Thus, $|\Q| \ge |\Q^*| \ge |\P^*| \ge \frac14|\P|$, as desired.
\end{proof}

\renewcommand{\K}{k}
\renewcommand{\w}{w}

\subsection{$3$-graph key argument}\label{subsec:key}


We next introduce a few more definitions that are needed for the statement of Lemma \ref{lemma:ind-k}.
Let $e(i) = 2^{i+10}$. We define the following tower-type function $\t\colon\N\to\N$;
\begin{equation}\label{eq:t}
\t(i+1) = \begin{cases}
2^{\t(i)/e(i)}	&\text{if } i \ge 1\\
2^{250}	&\text{if } i = 0 \;.
\end{cases}
\end{equation}
It is easy to prove, by induction on $i$, that $\t(i) \ge e(i)\t(i-1)$ for $i \ge 2$ (for the induction step,
$t(i+1) \ge 2^{\t(i-1)} = t(i)^{e(i-1)}$, so $t(i+1)/e(i+1) \ge \t(i)^{e(i-1)-i-11} \ge \t(i)$).
This means that $t$ is monotone increasing, and that $\t$ is an integer power of $2$ (follows by induction as $t(i)/e(i) \ge 1$ is a positive power of $2$ and in particular an integer).
We record the following facts regarding $\t$ for later use:
\begin{equation}\label{eq:monotone}
\t(i) \ge 4\t(i-1)  \quad\text{ and }\quad
\text{ $\t(i)$ is a power of $2$} \;.
\end{equation}
For a function $f:\N\to\N$ with $f(i) \ge i$ we denote
\begin{equation}\label{eq:f*}
f^*(i) = \t\big(f(i)\big)/e(i) \;.
\end{equation}
Note that $f^*(i)$ is indeed a positive integer (by the monotonicity of $\t$, $f^*(i) \ge \t(i)/e(i)$ is a positive power of $2$).
In fact, $f^*(i) \ge f(i)$ (as $f^*(i) \ge 4^{f(i)}/e(i)$ using~(\ref{eq:monotone})).
We recursively define the function $\w\colon\N\to\N$ as follows;
\begin{equation}\label{eq:Ak}
\w(i+1) = \begin{cases}
\w^*(i)		&\text{if } i \ge 1\\
1		&\text{if } i = 0 \;.
\end{cases}
\end{equation}
It is evident that $\w$ is a wowzer-type function; in fact, one can check that:
\begin{equation}\label{eq:A_k}
\w(i) \ge \wow(i) \;.
\end{equation}




\begin{lemma}[Key argument]\label{lemma:ind-k}
	Let $s \in \N$, let $\Vside^1,\Vside^2,\Vside^3$ be mutually disjoint sets of equal size and let $\V^1 \succ\cdots\succ \V^m$ be a sequence of
	$m=\w^*(s)+1$ successive equitable refinements of $\{\Vside^1,\Vside^2,\Vside^3\}$
	with $|\V^i_1|=|\V^i_2|=|\V^i_3|=\t(i)$ for every\footnote{Since we assume that each $\V^i$ refines $\{\Vside^1,\Vside^2,\Vside^3\}$ then $\V^i_1$ is (by the notation mentioned before Claim \ref{claim:star-union}) the restriction of $\V^i$ to $\Vside^1$.} $1 \leq i \leq m$.
	Then there is a $3$-partite $3$-graph $H$ on $(\Vside^1,\Vside^2,\Vside^3)$ of density $d(H)=2^{-s}$ satisfying the following property:\\
	If $(\Z,\E)$ is a $\langle 2^{-70} \rangle$-regular partition of $H$ and for some
	$1 \le i \le \w(s)$ $(< m)$ we have $\Z_3 \prec_{2^{-9}} \V^i_3$ and $\Z_2 \prec_{2^{-9}} \V^i_2$ then we also have $\Z_1 \prec_{2^{-9}} \V^{i+1}_1$.
\end{lemma}

\begin{proof}
	%
	%
	%
	%
	Put $s':=\w^*(s)$, so that $m = s'+1$.
	Apply Lemma~\ref{theo:core} with
	$$\Lside=\Vside^1,\quad \Rside=\Vside^2 \quad\text{ and }\quad \V^2_1 \succ \cdots \succ \V^{s'+1}_1 ,\quad \V^1_2 \succ \cdots \succ \V^{s'}_2 \;,$$
	and let
	\begin{equation}\label{eq:main-k-colors}
	\G^1 \succ \cdots \succ \G^{s'} \quad\text{ with }\quad |G^\ell|=2^\ell \text{ for every } 1 \le \ell \le s'
	\end{equation}
	be the resulting sequence of $s'$ successively refined equipartitions of $\Vside^1 \times \Vside^2$.
	%
	
	\begin{prop}\label{prop:main-k-hypo}
		Let $1 \le \ell \le s'$ and $G \in \G^\ell$.
		For every $\langle 2^{-28} \rangle$-regular partition $\Z_1 \cup \Z_2$ of $G$ (where $\Z_1$ and $\Z_2$ are partitions of $\Vside^1$ and $\Vside^2$, respectively) and every $1 \le i \le \ell$,
		if $\Z_2 \prec_{2^{-9}} \V^i_2$
		then $\Z_1 \prec_{2^{-9}} \V^{i+1}_1$.
	\end{prop}
	\begin{proof}
		First we need to verify that we may apply Lemma~\ref{theo:core} as above.
		Assumptions~\ref{item:core-minR},~\ref{item:core-expR} in Lemma~\ref{theo:core} hold by~(\ref{eq:monotone}) and the fact that $|\V^j_2|=\t(j)$.
		Assumption~\ref{item:core-expL} is satisfied since for every $1 \le j \le s$ we have
		$$|\V^{j+1}_1| = \t(j+1) = 2^{\t(j)/e(j)} = 2^{|\V^{j}_2|/e(j)} \;,$$
		where the second equality uses the definition of the function $\t$ in~(\ref{eq:t}).
		We can thus use Lemma~\ref{theo:core} to infer that the fact that $\Z_2 \prec_{2^{-9}} \V^i_2$ implies that $\Z_1 \prec_x \V^{i+1}_1$ with $x=\max\{2^{5}\sqrt{2^{-28}},\, 32/\sqrt[6]{\t(1)} \} = 2^{-9}$, using~(\ref{eq:t}).
	\end{proof}
	
	
	For each $1 \le j \le s$ let
	\begin{equation}\label{eq:main-k-dfns}
	\G^{(j)} = \G^{\w^*(j)}
	\quad\text{ and }\quad
	\V^{(j)} = \V^{\w(j)}_3 \;.
	\end{equation}	
	All these choices are well defined since $\w^*(j)$ satisfies $1 \le \w^*(1) \le \w^*(j) \le \w^*(s) = s'$, and since $\w(j)$ satisfies $1 \le \w(1) \le \w(j) \le \w(s) \le m$. Observe that we have thus chosen two subsequences of $\G^1,\cdots,\G^{s'}$ and $\V^1_3,\ldots,\V^m_3$, each of length $s$.
	Recalling that each $\G^{(j)}$ is a partition of $\Vside^1 \times \Vside^2$, we now apply Lemma~\ref{theo:core} again with
	$$
	\Lside=\Vside^1 \times \Vside^2,\quad \Rside=\Vside^3 \quad\text{ and }\quad \G^{(1)} \succ \cdots \succ \G^{(s)}, \quad \V^{(1)} \succ \cdots \succ \V^{(s)} \;.
	$$
	The output of this application of Lemma~\ref{theo:core} consists of a sequence of $s$ (successively refined)
	equipartitions of $(\Vside^1 \times \Vside^2)\times\Vside^3$.
	We can think of the $s$-th partition of this sequence as a collection of $2^s$ bipartite graphs on vertex sets
	$(\Vside^1\times\Vside^{2},\,\Vside^3)$. For the rest of the proof let $G'$ be be any of these graphs.
	By Remark \ref{remequi} we have
	\begin{equation}\label{eq:ind-colors2}
	d(G')=2^{-s} \;.
	\end{equation}
	\begin{prop}\label{prop:ind-prop2}
		For every $\langle 2^{-70} \rangle$-regular partition $\E \cup \V$ of $G'$ (where $\E$ and $\V$ are partitions of $\Vside^1\times\Vside^{2}$ and $\Vside^3$ respectively)
		and every $1 \le j' \le s$,
		if $\V \prec_{2^{-9}} \V^{(j')}$ then $\E \prec_{2^{-30}} \G^{(j')}$.
	\end{prop}
	\begin{proof}
		First we need to verify that we may apply Lemma~\ref{theo:core} as above.
		Note that $|\G^{(j)}|=2^{\w^*(j)}$ by~(\ref{eq:main-k-colors}) and (\ref{eq:main-k-dfns}),
		and that $|\V^{(j)}|=\t(\w(j))$ by (\ref{eq:main-k-dfns}) and the statement's assumption that $|\V^i_3|=\t(i)$.
		Therefore,
		\begin{equation}\label{eq:main-k-orders}
		|\G^{(j)}| = 2^{\w^*(j)} = 2^{\t(\w(j))/e(j)} = 2^{|\V^{(j)}|/e(j)} \;,
		\end{equation}
		where the second equality relies on~(\ref{eq:f*}).
		Moreover, note that $\t(\w(1)) = \t(1) = 2^{300}$.
		Now, Assumptions~\ref{item:core-minR} and~\ref{item:core-expR} in Lemma~\ref{theo:core} follow from the fact that $|\V^{(j)}|=\t(\w(j))$, from~(\ref{eq:monotone})
		and the fact that $|\V^{(1)}| = \t(\w(1)) \ge 2^{200}$ by~(\ref{eq:Ak}). 	
		Assumption~\ref{item:core-expL} follows from~(\ref{eq:main-k-orders}).
		We can thus use Lemma~\ref{theo:core} to infer that the fact that $\V \prec_{2^{-9}} \V^{(j')}$ implies that
		$\E \prec_x \G^{(j')}$ with $x=\max\{2^{5}\sqrt{2^{-70}},\, 32/\sqrt[6]{\t(\w(1))} \} = 2^{-30}$.
	\end{proof}
	
	Let $H$ be the $3$-partite $3$-graph on vertex classes $(\Vside^1,\Vside^2,\Vside^3)$ with edge set
	$$
	E(H) = \big\{ (v_1,v_2,v_3) \,:\, ((v_1,v_{2}),v_3) \in E(G') \big\} \;,
	$$	
	and note that we have (recall Definition \ref{def:aux})
	\begin{equation}\label{eqH}
	G'=G_{H}^3\;.
	\end{equation}
	We now prove that $H$ satisfies the properties in the statement of the lemma.

	%
	First, note that by~(\ref{eq:ind-colors2}) and (\ref{eqH}) we have $d(H)=2^{-s}$, as needed.
	Assume now that $i$ is such that
	\begin{equation}\label{eq:ind-i-assumption}
	1 \le i \le \w(s)
	\end{equation}
	and:
	\begin{enumerate}
		\item\label{item:ind-reg}
		$(\Z,\E)$ is a $\langle 2^{-70} \rangle$-regular partition of $H$, and
		\item\label{item:ind-refine} $\Z_3 \prec_{2^{-9}} \V^i_3$ and $\Z_2 \prec_{2^{-9}} \V^i_2$.
	\end{enumerate}	
	We need to show that
	\begin{equation}\label{eq:ind-goal}
	\Z_1 \prec_{2^{-9}} \V^{i+1}_1 \;.
	\end{equation}
	Since Item~\ref{item:ind-reg} guarantees that $(\Z,\E)$ is a $\langle 2^{-70} \rangle$-regular partition of $H$,
	we get from Definition~\ref{def:k-reg} and (\ref{eqH}) that
	\begin{equation}\label{eq:ind-reg}
	\text{$\E_3 \cup \Z_3$ is a $\langle 2^{-70} \rangle$-regular partition of } G'.
	\end{equation}
	Let
	\begin{equation}\label{eq:ind-j'}
	1 \le j' \le s
	\end{equation}
	be the unique integer satisfying (the equality here is just (\ref{eq:Ak}))
	\begin{equation}\label{eq:ind-sandwich}
	\w(j') \le i < \w(j'+1) = \w^*(j')\;.
	\end{equation}
	Note that (\ref{eq:ind-j'}) holds due to~(\ref{eq:ind-i-assumption}).
	Recalling~(\ref{eq:main-k-dfns}),
	the lower bound in~(\ref{eq:ind-sandwich}) implies that $\V^i_3 \prec \V^{\w(j')} = \V^{(j')}$.
	Therefore, the assumption $\Z_3 \prec_{2^{-9}} \V^i_3$ in~\ref{item:ind-refine} implies that
	\begin{equation}\label{eq:ind-Zk}
	\Z_3 \prec_{2^{-9}} \V^{(j')} \;.
	\end{equation}
	Apply Proposition~\ref{prop:ind-prop2} on $G'$, using~(\ref{eq:ind-reg}),~(\ref{eq:ind-j'}) and~(\ref{eq:ind-Zk}), to deduce that
	\begin{equation}\label{eq:ind-E}
	\E_3 \prec_{2^{-30}} \G^{(j')} = \G^{\w^*(j')} \;,
	\end{equation}
	where for the equality again recall~(\ref{eq:main-k-dfns}).
	Since $(\Z,\E)$ is a $\langle 2^{-70} \rangle$-regular partition of $H$ (by Item~\ref{item:ind-reg} above)
	it is in particular $\langle 2^{-70} \rangle$-good. By~(\ref{eq:ind-E}) we may thus apply Claim~\ref{claim:uniform-refinement}
	to conclude that
	\begin{equation}\label{eq:ind-reg2}
	\Z_1 \cup \Z_2 \text{ is a }
	\langle 2^{-28} \rangle
	\text{-regular partition of some $G\in\G^{\w^*(j')}$.}
	\end{equation}
	By~(\ref{eq:ind-reg2}) we may apply Proposition~\ref{prop:main-k-hypo} with $G$, $\Z_1\cup\Z_2$, $\ell=\w^*(j')$ and $i$, observing (crucially)
	that $i \leq \ell$ by (\ref{eq:ind-sandwich}). We thus conclude that the fact $\Z_2 \prec_{2^{-9}} \V^i_2$ (stated in~\ref{item:ind-refine})
	implies that $\Z_1 \prec_{2^{-9}} \V^{i+1}_1$, thus proving~(\ref{eq:ind-goal}) and completing the proof.
\end{proof}

\subsection{Putting everything together}\label{subsec:pasting}


%
%
We can now prove our main theorem, Theorem~\ref{theo:main}, which we repeat here for convenience.
\addtocounter{theo}{-2}
\begin{theo}[Main theorem]
	Let $s \in \N$.
	There exists a $3$-partite $3$-graph $H$ on vertex classes of equal size and of density at least $2^{-s}$,
	and a partition $\V_0$ of $V(H)$ with $|\V_0| \le 2^{300}$, such that if
	$(\Z,\E)$ is a $\langle 2^{-73} \rangle$-regular partition of $H$ with
	$\Z \prec \V_0$ then $|\Z| \ge \wow(s)$.
\end{theo}
\addtocounter{theo}{+1}



\begin{proof}
	%
	Let the $3$-graph $B$ be the tight $6$-cycle; that is, $B$ is the $3$-graph on vertex classes $\{0,1,\ldots,5\}$ with edge set $E(B)=\{\{0,1,2\},\{1,2,3\},\{2,3,4\},\{3,4,5\},\{4,5,0\},\{5,0,1\}\}$.
	Note that $B$ is $3$-partite with vertex classes $(\{0,3\},\{1,4\},\{2,5\}\}$.
	Put $m=\w^*(s-1)+1$ and let $n \ge \t(m)$.
	Let $\Vside^0,\ldots,\Vside^{5}$ be $6$ mutually disjoint sets of size $n$ each.
	Let $\V^1 \succ\cdots\succ \V^m$ be an arbitrary sequence of $m$ successive equitable refinements of $\{\Vside^0,\ldots,\Vside^{5}\}$ with $|\V^i_h|=\t(i)$ for every $1 \le i \le m$ and $0 \le h \le 5$, which exists as $n$ is large enough.
	Extending the notation $\Z_i$ (above Definition~\ref{def:k-reg}), for every $0 \le x \le 5$
	we henceforth denote the restriction of the vertex partition $\Z$ to $\Vside^x$ by $\Z_x = \{Z \in \Z \,\vert\, Z \sub \Vside^x\}$.	
	For each edge $e=\{x,x+1,x+2\} \in E(B)$ (here and henceforth when specifying an edge, the integers are implicitly taken modulo $6$)
	apply Lemma~\ref{lemma:ind-k} with
	$$s-1,\,
	\Vside^{x},\Vside^{x+1},\Vside^{x+2}
	\text{ and }
	(\V^{1}_x \cup \V^1_{x+1} \cup \V^1_{x+2})
	\succ\cdots\succ (\V^{m}_{x}\cup\V^{m}_{x+1}\cup\V^{m}_{x+2}) \;.$$
	Let $H_e$ denote the resulting $3$-partite $3$-graph on $(\Vside^{x},\Vside^{x+1},\Vside^{x+2})$.
	Note that $d(H_e) = 2^{-(s-1)}$.
	Moreover, let
	$$c = 2^{-9} \quad\text{ and }\quad K=\w(s-1)+1 \;.$$
	Then $H_e$ has the property that for every $\langle 2^{-70} \rangle$-regular partition $(\Z',\E')$ of $H_e$ and every $1 \le i < K$,
	%
	%
	\begin{equation}\label{eq:paste-property}
	\text{if $\Z'_{x+2} \prec_{c} \V^i_{x+2}$ and $\Z'_{x+1} \prec_{c} \V^i_{x+1}$ then $\Z'_x \prec_{c} \V^{i+1}_x$.}
	\end{equation}	
	We construct our $3$-graph on the vertex set $\Vside:=\Vside^0 \cup\cdots\cup \Vside^5$ as
	$E(H) = \bigcup_{e} E(H_e)$; that is, $H$ is the edge-disjoint union of all six $3$-partite $3$-graphs $H_e$ constructed above.
	Note that $H$ is a $3$-partite $3$-graph (on vertex classes $(\Vside^0 \cup \Vside^3,\, \Vside^1 \cup \Vside^4,\, \Vside^2 \cup \Vside^5))$ of density $\frac68 2^{-(s-1)} \ge 2^{-s}$, as needed.
	%
	%
	%
	%
	%
	%
	%
	We will later use the following fact.
	\begin{prop}\label{prop:restriction}
		Let $(\Z,\E)$ be an $\langle 2^{-73}\rangle$-regular partition of $H$ and let
		$e \in E(B)$.
		If $\Z \prec \{\Vside^0,\ldots,\Vside^{5}\}$
		then the restriction $(\Z',\E')$ of $(\Z,\E)$ to $V(H_e)$ is a $\langle 2^{-70} \rangle$-regular partition of $H_e$.
	\end{prop}
	\begin{proof}
		Immediate from Claim~\ref{claim:restriction} using the fact that $e(H_e) = \frac16 e(H)$.
		%
	\end{proof}	
	
	Now, let $(\Z,\E)$ be a $\langle 2^{-73} \rangle$-regular partition of $H$
	with $\Z \prec \V^1$.
	Our goal will be to show that
	\begin{equation}\label{eq:paste-goal}
	\Z \prec_{c} \V^{K} \;.
	\end{equation}
	Proving~(\ref{eq:paste-goal}) would complete the proof, by setting $\V_0$ in the statement to be $\V^1$ here (notice $|\V^1|=3\t(1) \le 2^{300}$ by~(\ref{eq:t})); 
	indeed, Claim~\ref{claim:refinement-size} would imply that
	$$|\Z| \ge \frac14|\V^{K}| = \frac14 \cdot 6 \cdot \t(K)
	\ge \t(K)
	\ge \t(\w(s-1))
	\ge \w(s)
	\ge \wow(s) \;,$$
	where the last inequality uses~$(\ref{eq:A_k})$.
	Assume towards contradiction that $\Z \nprec_{c} \V^{K}$. By averaging,
	\begin{equation}\label{eq:assumption}
	\Z_h \nprec_c \V^{K}_h \text{ for some } 0 \le h \le 5.
	\end{equation}
	For each $0 \le h \le 5$ let $1 \le \b(h) \le K$ be the largest integer satisfying $\Z_h \prec_c \V^{\b(h)}_h$,
	which is well defined since $\Z_h \prec_c \V^1_h$,
	since in fact $\Z \prec \V^1$.
	Put $\b^* = \min_{0 \le h \le 5} \b(h)$, and note that by~(\ref{eq:assumption}),
	\begin{equation}\label{eq:paste-star}
	\b^* < K \;.
	\end{equation}
	Let $0 \le x \le 5$ minimize $\b$, that is, $\b(x)=\b^*$.
	Therefore:
	\begin{equation}\label{eqcontra}
	\Z_{x+2} \prec_c \V^{\b^*}_{x+2} \mbox{~,~} \Z_{x+1} \prec_c \V^{\b^*}_{x+1} \mbox{ and }
	\Z_{x} \nprec_c \V^{\b^*+1}_{x}.		
	\end{equation}
	Let $e=\{x,x+1,x+2\} \in E(B)$.
	Let $(\Z',\E')$ be the restriction of $(\Z,\E)$ to $V(H_e)=\Vside^{x} \cup \Vside^{x+1} \cup \Vside^{x+2}$, which is a $\langle 2^{-70} \rangle$-regular partition of $H_e$ by Proposition~\ref{prop:restriction}. Since $\Z'_x=\Z_x$, $\Z'_{x+1}=\Z_{x+1}$, $\Z'_{x+2}=\Z_{x+2}$ we get
	from (\ref{eqcontra}) a contradiction to~(\ref{eq:paste-property}) with $i=\beta^*$.
	We have thus proved~(\ref{eq:paste-goal}) and so the proof is complete.
\end{proof}

\section{Wowzer-type Lower Bounds for $3$-Graph Regularity Lemmas}\label{sec:FR}

\renewcommand{\K}{\mathcal{K}}

The purpose of this section is to apply Theorem \ref{theo:main} in order to prove Corollary~\ref{coro:FR-LB},
thus giving wowzer-type (i.e., $\Ack_3$-type) lower bounds for the $3$-graph regularity lemmas of Frankl and R\"odl~\cite{FrankRo02} and of Gowers~\cite{Gowers06}.
We will start by giving the necessary definitions for Frankl and R\"odl's lemma
and state our corresponding lower bound.
Next we will state the necessary definitions for Gower's lemma and state our corresponding lower bound.
The formal proofs would then follow.



\subsection{Frankl and R\"odl's $3$-graph regularity}





\begin{definition}[$(\ell,t,\e_2)$-equipartition,~\cite{FrankRo02}]\label{def:ve-partition}
%
	An \emph{$(\ell,t,\e_2)$-equipartition} on a set $V$ is a $2$-partition $(\Z,\E)$ on $V$ where $\Z$ is an equipartition of order $|\Z|=t$ and every graph in $\E$ is $\e_2$-regular\footnote{Here, and in several places in this section, we of course refer to the ``traditional'' notion of Szemer\'edi's $\e$-regularity, as defined at the beginning of Section \ref{sec:define}. } of density $\ell^{-1} \pm \e_2$.
%
\end{definition}

\begin{remark}
	If $\e_2 \le \frac12\ell^{-1}$ then $(Z,\E)$ has at most $2\ell$ bipartite graphs between every pair of clusters of $\Z$.
\end{remark}


A \emph{triad} of a $2$-partition $(\Z,\E)$ is any tripartite graph whose three vertex classes are in $\Z$ and three edge sets are in $\E$. We often identify a triad with a triple of its edge sets $(E_1,E_2,E_3)$.
The \emph{density} of a triad $P$ in a $3$-graph $H$ is $d_H(P)=|E(H) \cap T(P)|/|T(P)|$ (and $0$ if $|T(P)|=0$).
A \emph{subtriad} of $P$ is any subgraph of $P$ on the same vertex classes.


\begin{definition}[$3$-graph $\e$-regularity~\cite{FrankRo02}]\label{def:FR-reg}
	Let $H$ be a $3$-graph.
	A triad $P$ is \emph{$\e$-regular} in $H$ if every subtriad $P'$ with $|T(P')| \ge \e|T(P)|$ satisfies $|d_H(P')-d_H(P)| \le \e$.\\
	An $(\ell,t,\e_2)$-equipartition $\P$ on $V(H)$ is an \emph{$\e$-regular} partition of $H$ if
	$\sum_P |T(P)| \le \e|V|^3$ where the sum is over all triads of $\P$ that are not $\e$-regular in $H$.	
\end{definition}

The $3$-graph regularity of Frankl and R\"odl~\cite{FrankRo02} states, very roughly, that for every $\e>0$ and every function $\e_2\colon\N\to(0,1]$, every $3$-graph has an $\e$-regular $(\ell,t,\e_2(\ell))$-equipartition where $t,\ell$ are bounded by a wowzer-type function.
In fact, the statement in~\cite{FrankRo02} uses a considerably stronger notion of regularity of a partition than in Definition~\ref{def:FR-reg} that involves an additional function $r(t,\ell)$ which we shall not discuss here (as discussed in~\cite{FrankRo02}, this stronger notion was crucial for allowing them to prove the $3$-graph removal lemma).
Our lower bound below applies even to the weaker notion stated above, which corresponds to taking $r(t,\ell)\equiv 1$.


Using Theorem~\ref{theo:main} we can deduce a wowzer-type \emph{lower} bound for Frankl and R\"odl's $3$-graph regularity lemma.
The proof of this lower bound appears in Subsection~\ref{subsec:FR-LB-proof}.



\begin{theo}[Lower bound for Frankl and R\"odl's regularity lemma]\label{theo:FR-LB}
	Put $c = 2^{-400}$. 
	For every $s \in \N$ there exists a $3$-partite $3$-graph $H$ of density $p=2^{-s}$, and a partition $\V_0$ of $V(H)$ with $|\V_0| \le 2^{300}$,
	such that
	if $(\Z,\E)$ is an $\e$-regular $(\ell,t,\e_2(\ell))$-equipartition of $H$,
	with $\e \le c p$, $\e_2(\ell) \le c \ell^{-3}$ and $\Z \prec \V_0$, then $|\Z| \ge \wow(s)$.
\end{theo}

\begin{remark}
	One can easily remove the assumption $\Z \prec \V_0$ by taking the common refinement of $\Z$ with $\V_0$ (and adjusting $\E$ appropriately).
Since $|\V_0|=O(1)$ this has only a minor effect on the parameters $\e,\ell,t,\e_2(\ell)$ of the partition and thus
one gets essentially the same lower bound. We omit the details of this routine transformation.
\end{remark}

\subsection{Gowers' $3$-graph regularity}

Here we consider the $3$-graph regularity Lemma due to Gowers~\cite{Gowers06}.


\begin{definition}[$\a$-quasirandomness, see Definition~6.3 in \cite{Gowers06}]\label{def:quasirandom}
	Let $H$ be a $3$-graph, and let $P=(E_0,E_1,E_2)$ be a triad with $d(E_0)=d(E_1)=d(E_2)=:d$ on vertex classes $(X,Y,Z)$ with $|X|=|Y|=|Z|=:n$. We say that $P$ is \emph{$\a$-quasirandom} in $H$ if
	$$\sum_{x_0,x_1 \in X}\sum_{y_0,y_1 \in Y}\sum_{z_0,z_1 \in Z} \prod_{i,j,k\in\{0,1\}} f(x_i,y_j,z_k) \le \a d^{12}n^6 \;,$$
	where
	$$f(x,y,z) =
	\begin{cases}
	1-d_H(P)				&\text{if } (x,y,z) \in T(P), (x,y,z) \in E(H)\\
	-d_H(P) 				&\text{if } (x,y,z) \in T(P), (x,y,z) \notin E(H)\\
	0 						&\text{if } (x,y,z) \notin T(P) \;.
	\end{cases}$$
	An $(\ell,t,\e_2)$-equipartition $\P$ on $V(H)$ is an \emph{$\a$-quasirandom} partition of $H$ if
	$\sum_P |T(P)| \le \a|V|^3$ where the sum is over all triads of $\P$ that are not $\a$-quasirandom in $H$.	
\end{definition}

The $3$-graph regularity lemma of Gowers~\cite{Gowers06} (see also~\cite{NaglePoRoSc09}) can be equivalently phrased as stating that, very roughly, for every $\a>0$ and every function $\e_2\colon\N\to(0,1]$, every $3$-graph has an $\a$-quasirandom $(\ell,t,\e_2(\ell))$-equipartition where $t,\ell$ are bounded by a wowzer-type function.

One way to prove a wowzer-type lower bound for Gowers' $3$-graph regularity lemma is along similar lines to the proof of Theorem~\ref{theo:FR-LB}. 
However, there is shorter proof using the fact that Gowers' notion of quasirandomness implies Frankl and R\"odl's notion of regularity.
In all that follows we make the rather trivial assumption that, in the notation above, $\a,1/\ell \le 1/2$.

\begin{prop}[\cite{NagleRoSc17}]\label{prop:Schacht}
	There is $C \ge 1$ such that the following holds;
	if a triad $P=(E_0,E_1,E_2)$ is $\e^C$-quasirandom and for every $0 \le i \le 2$ the bipartite graph $E_i$ is $d(E_i)^C$-regular then $P$ is $\e$-regular.	
%
	%
	%
\end{prop}

Our lower bound for Gowers' $3$-graph regularity lemma is as follows.

\begin{theo}[Lower bound for Gowers' regularity lemma]\label{theo:Gowers-LB}
	For every $s \in \N$ there exists a $3$-partite $3$-graph $H$ of density $p=2^{-s}$, and a partition $\V_0$ of $V(H)$ with $|\V_0| \le 2^{300}$,
	such that
	if $(\Z,\E)$ is an $\a$-quasirandom $(\ell,t,\e_2(\ell))$-equipartition of $H$,
	with $\a \le \poly(p)$, $\e_2(\ell) \le \poly(1/\ell)$ and $\Z \prec \V_0$, then $|\Z| \ge \wow(s)$.
\end{theo}


\begin{proof}
	Given $s$, let $H$ and $\V_0$ be as in Theorem~\ref{theo:FR-LB}.
	Let $\P=(\Z,\E)$ be an $\a$-quasirandom $(\ell,t,\e_2(\ell))$-equipartition of $H$ with $\Z \prec \V_0$, $\a \le (cp)^C$ and $\e_2(\ell) \le \min\{c\ell^{-3},\,(2\ell)^{-C}\}$, where $c$ and $C$ are as in Theorem~\ref{theo:FR-LB} and Proposition~\ref{prop:Schacht} respectively.
	We will show that $\P$ is a $cp$-regular partition of $H$,
	which would complete the proof using Theorem~\ref{theo:FR-LB} and the fact that $\e_2 \le c\ell^{-3}$.
	Let $P=(E_0,E_1,E_2)$ be a triad of $\P$ that is $\a$-quasirandom in $H$.
	Note that, by our choice of $\e_2(\ell)$, for every $0 \le i \le 2$ we have $d(E_i) \ge 1/\ell - \e_2(\ell) \ge 1/2\ell$; thus, since $\e_2(\ell) \le (1/2\ell)^{C} \le d(E_i)^C$, we have that $E_i$ is $d(E_i)^C$-regular.
	Applying Proposition~\ref{prop:Schacht} on $P$ we deduce that $P$ is $\e$-regular with $\e=\a^{1/C} \le cp$.
	Since $\P$ is an $\a$-quasirandom partition of $H$ we have, by Definition~\ref{def:quasirandom} and since $\a \le \e$, that $\P$ is an $\e$-regular partition of $H$, as needed.
	%
%
\end{proof}

\subsection{Proof of Theorem~\ref{theo:FR-LB}}\label{subsec:FR-LB-proof}

The proof of Theorem~\ref{theo:FR-LB} will follow quite easily from Theorem~\ref{theo:main} together with Claim~\ref{claim:reduction} below.
Claim~\ref{claim:reduction} basically shows that a $\langle \d \rangle$-regularity ``analogue'' of Frankl and R\"odl's notion of regularity implies graph $\langle \d \rangle$-regularity. 
%
%
%
%
Here it will be convenient to say that a graph partition is \emph{perfectly} $\langle \d \rangle$-regular if all pairs of distinct clusters are $\langle \d \rangle$-regular without modifying any of the graph's edges.
Furthermore, we will henceforth abbreviate $t(P)=|T(P)|$ for a triad $P$.
We will only sketch the proof of Claim~\ref{claim:reduction}, deferring the full details to the Appendix~\ref{sec:FR-appendix}.

\begin{claim}\label{claim:reduction}
	Let $H$ be a $3$-partite $3$-graph on vertex classes $(\Aside,\Bside,\Cside)$,
	and let $(\Z,\E)$ be an $(\ell,t,\e_2)$-equipartition of $H$ with $\Z \prec \{\Aside,\Bside,\Cside\}$ such that for every triad $P$ of $\P$ and every subtriad $P'$ of $P$ with $t(P') \ge \d \cdot t(P)$ we have $d_H(P') \ge \frac23 d_H(P)$.
	If $\e_2(\ell) \le (\d^2/88)\ell^{-3}$
	then $\E_3 \cup \Z_3$ is a perfectly $\langle 2\sqrt{\d} \rangle$-regular partition of $G_H^3$.
\end{claim}

\begin{proof}[Proof (sketch):]
We remind the reader that the vertex classes of $G_H^3$ are $(\Aside \times \Bside,\, \Cside)$ (recall Definition~\ref{def:aux}), and that $\E_3$ and $\Z_3$ are the partition of $\Aside\times\Bside$ induced by $\E$ and the partition of $\Cside$ induced by $\Z$, respectively. Suppose $(\Z,\E)$ is as in the statement of the claim, and define $\E'$ as follows:
for every $A \in \Z_1$ and $C \in \Z_3$, replace all the bipartite graphs between $A$ and $C$ with the complete bipartite graph $A \times C$.
Do the same for every $B \in \Z_2$ and $C \in \Z_3$ (we do {\em not} change the partitions between $\Aside$ and $\Bside$). The simple (yet somewhat tedious to prove)
observation is that if all triads of $(\Z,\E)$ are regular then all triads of $(\Z,\E')$ are essentially as regular.
Once this observation is proved, the proof of the claim reduces to checking definitions. We thus defer the proof to Appendix~\ref{sec:FR-appendix}.
\end{proof}

%
%
Using Claim~\ref{claim:reduction}, we now prove our wowzer lower bound.

\begin{proof}[Proof of Theorem~\ref{theo:FR-LB}]
	Put $\a = 2^{-73}$.
	We have
	\begin{equation}\label{eq:FR-LB-ineq}
	c = 2^{-400} \le \a^4/1500 \;.
	\end{equation}
	Given $s$, let $H$ and $\V_0$ be as in Theorem~\ref{theo:main}.
	Let $\P=(\Z,\E)$ be an $\e$-regular $(\ell,t,\e_2(\ell))$-equipartition of $H$
	with $\e \le c p$, $\e_2(\ell) \le c \ell^{-3}$ and $\Z \prec \V_0$.
	Thus, the bound $|Z| \ge \wow(s)$ would follow from Theorem~\ref{theo:main}
	if we show that $\P$ is an  $\langle \a \rangle$-regular partition of $H$.
	First we need to show that $\P$ is $\langle \a \rangle$-good (recall Definition~\ref{def:k-good}). Let $E$ be a graph with $E \in \E$ on vertex classes $(Z,Z')$ (so $Z \neq Z' \in \Z$).
	We need to show that $E$ is $\langle \a \rangle$-regular.
	Since $\P$ is an $(\ell,t,\e_2(\ell))$-equipartition we have
	(recall Definition~\ref{def:ve-partition}) that $E$ is $\e_2(\ell)$-regular and $d(E) \ge \ell^{-1} - \e_2(\ell)$.
	The statement's assumption on $\e_2(\ell)$ thus implies
	$d(E) \ge 2\e_2(\ell)$.
	It follows that for every $S \sub Z$, $S' \sub Z'$ with $|S| \ge \e_2(\ell)|Z|$, $|S'| \ge \e_2(\ell)|Z'|$ we have $d_E(S,S') \ge d(E)-\e_2(\ell) \ge \frac12 d(E)$.
	This proves that $E$ is $\langle \e_2(\ell) \rangle$-regular, and since $\e_2(\ell) \le c \le \a$, that $E$ is $\langle \a \rangle$-regular, as needed.
	
	
	%
	%
	It remains to show that the $\langle \a \rangle$-good $\P$ is an $\langle \a \rangle$-regular partition of $H$ (recall Definition~\ref{def:k-reg}).
	By symmetry, it suffices to show that $\E_3 \cup \Z_3$ is an $\langle \a \rangle$-regular partition of $G_{H}^3$.
	Let $H'$ be obtained from $H$ by removing all ($3$-)edges in triads of $\P$ that are either not $\e$-regular in $H$ or have density at most $3\e$ in $H$.
	By Definition~\ref{def:FR-reg}, the number of edges removed from $H$ to obtain $H'$ is at most
	\begin{equation}\label{eq:FR-LB-modify}
	\e|V(H)|^3 + 3\e|V(H)|^3 \le 4\cdot c p |V(H)|^3
	\le (\a p/27)|V(H)|^3 = \a\cdot e(H) \;,
	\end{equation}
	where the second inequality uses ~(\ref{eq:FR-LB-ineq}),
	and the equality uses the fact that all three vertex classes of $H$ are of the same size. 
	Thus, in $H'$, every non-empty triad of $\P$ is $\e$-regular and of density at least $3\e$.
	Put $\d = (\a/2)^2$.
	Again by Definition~\ref{def:FR-reg}, for every triad $P$ of $\P$ and every subtraid $P'$ of $P$ with $t(P') \ge \d \cdot t(P)$ ($\ge \e \cdot t(P)$ by~(\ref{eq:FR-LB-ineq})) we have $d_{H'}(P') \ge d_{H'}(P)-\e \ge \frac23 d_{H'}(P)$.
	It follows from applying Claim~\ref{claim:reduction} with $H'$ and $\d$,
	using~(\ref{eq:FR-LB-ineq}),
	that $\E_3 \cup \Z_3$ is a perfectly $\langle \a \rangle$-regular partition of $G_{H'}^3$.
	Note that~(\ref{eq:FR-LB-modify}) implies that
	one can add/remove at most $\a \cdot e(G_{H}^3)$ edges of $G_{H}^3$ to obtain $G_{H'}^3$.
	Thus, $\E_3 \cup \Z_3$ is an $\langle \a \rangle$-regular partition of $G_{H}^3$, and as explained above, this completes the proof.
\end{proof}

\appendix

\section{Proof of Claim~\ref{claim:reduction}}\label{sec:FR-appendix}

Our purpose here is to prove Claim~\ref{claim:reduction}.
First, we will need some standard facts about $\e$-regular graphs.
We begin with the so-called \emph{slicing lemma} for graphs.

\begin{claim}
	\label{claim:slice2}
	Let $\a \ge \e > 0$.
	Let $(A,B)$ be an $\e$-regular pair of density $d$.
	If $A' \sub A$, $B' \sub B$ are of size $|A'| \ge \a|A|$, $|B'| \ge \a|B|$ then the pair $(A',B')$ is
	$2\e/\a$-regular of density $d \pm \e$.
\end{claim}
\begin{proof}
	First, $|d(A',B')-d| \le \e$ is immediate as $G$ is $\e$-regular and $\a \ge \e$.
	Next, if $X \sub A'$ and $Y \sub B'$ satisfy $|X| \ge (\e/\a)|A'|$ and $|Y| \ge (\e/\a)|B'|$ then 
	$|X| \ge \e|A|$ and $|Y| \ge \e|B|$.
	Since $(A,B)$ is $(\e,d)$-regular we have
	$|d(X,Y)-d(A',B')| \le |d(X,Y)-d|+|d-d(A',B')| \le 2\e \le 2\e/\a$.
\end{proof}

Next, a well-known fact regarding the degrees of most vertices in a regular bipartite graph.

\begin{claim}\label{claim:degrees}
	For every $\e$-regular bipartite graph on $(X,Y)$,
	if $Y' \sub Y$ satisfies $|Y'| \ge \e|Y|$ then
	all vertices of $X$ but at most $2\e|X|$ have degree $(d(X,Y) \pm \e)|Y'|$ into $Y'$.
\end{claim}
\begin{proof}
	Otherwise, there is a set $X' \sub X$ of $\e|X|$ vertices whose degrees into $Y'$ are all, without loss of generality, greater than $(d(X,Y) + \e)|Y'|$. Thus $d(X',Y') > d(X,Y) + \e$, a contradiction.
\end{proof}

Finally, a proof of a version of the \emph{triangle counting lemma}
that we will need to prove Claim~\ref{claim:reduction}. Crucially, we will rely on the fact that one of the three bipartite graphs is not required to be regular.

\begin{lemma}[Triangle counting lemma]\label{lemma:triangle-counting}
	For every triad $P$ on vertex classes $(A,B,C)$,
	if both $(A,C)$ and $(B,C)$ are $\e$-regular
	then
	$$t(P) = \Big( d(A,B)d(A,C)d(B,C) \pm 7\e \Big) |A||B||C| \;.$$
\end{lemma}
\begin{proof}
	Let
	$F = \{ (a,b) \in A \times B \,:\, \codeg(a,b) \neq (d(A,C)d(B,C) \pm 3\e)|C| \}$.
	To bound $|F|$, first apply Claim~\ref{claim:degrees} on $(A,C)$ with $Y'=C$. Thus, there is a subset $A^* \sub A$ with $|A^*| \le 2\e|A|$ such that every $a \notin A^*$ satisfies $e(a,C) = (d(A,C) \pm \e)|C|$.
	Let $a \in A \sm A^*$ and apply Claim~\ref{claim:degrees}, this time on $(B,C)$ with $Y'=E(a,C)$, and note that $E(a,C) \sub C$ satisfies $e(a,C) \ge (d(A,C)-\e)|C| \ge \e|C|$ as required (for the last inequality we assume $d(A,C) \ge 2\e$, as otherwise there is nothing to prove).
	Thus, there is a subset $B^*_a \sub B$ with $|B^*_a| \le 2\e|B|$ such that every $b \notin B^*_a$ satisfies
	$$\codeg(a,b) = (d(B,C) \pm \e)\cdot (d(A,C) \pm \e)|C|
	= (d(A,C)d(B,C) \pm 3\e)|C| \;.$$
	Therefore,
	$|F| \le |A^*||B| + \sum_{a \in A \sm A^*}|B^*_a| \le 4\e|A||B|$.
	We deduce that
	\begin{align*}
	t(P) &= \sum_{(a,b) \in E(A,B)} \codeg(a,b) =
	e(A,B) \cdot (d(A,C)d(B,C) \pm 3\e)|C| \pm |F||C| \\
	&= d(A,B) \cdot (d(A,C)d(B,C) \pm 3\e) |A||B||C| \pm 4\e|A||B||C|\;,
	\end{align*}
	which completes the proof.
	%
	%
\end{proof}

\subsection{Proof of Claim~\ref{claim:reduction}}


%
%


\begin{proof}
	Put $G=G_H^3$,
	and let $E \in \E_3$ and $C \in \Z_3$.
	Recall that we can also view $E$ as vertex subset of $G$.
	Our goal is to prove that $G[E,C]$ is 
	$\langle \d' \rangle$-regular with $\d'=2\sqrt{\d}$.
	Fix $E' \sub E$ and $C' \sub C$ with $|E'| \ge \d'|E|$ and $|C'| \ge \d'|C|$.
	Thus, our goal is to prove that
	$d_{G}(E',C') \ge \frac12 d_{G}(E,C)$.
	Suppose that $E$ lies between $A \in \Z_1$ and $B \in \Z_2$, and consider the following triad and subtriad;
	$$P=(E,\,A \times C,\, B \times C)\quad \text{ and }\quad Q=(E',\,A \times C',\,B \times C') \;.$$
	We claim that
	\begin{equation}\label{eq:red-d}
	d_{G}(E,C) = d_{H}(P) \quad\text{ and }\quad d_{G}(E',C') = d_{H}(Q) \;.
	\end{equation}

	Note that proving~(\ref{eq:red-d}) would mean that to complete the proof it suffices to show that
	\begin{equation}\label{eq:red-goal}
	d_H(Q) \ge \frac12 d_H(P)  \;.
	\end{equation}
	To prove~(\ref{eq:red-d}) first note that the set of triangles in $Q$ is given by
	\begin{align*}
	T(Q) &= \{ \,(a,b,c) \colon (a,b) \in E',\, (a,c) \in A \times C',\, (b,c) \in B \times C' \,\} \\
	&= \{ \,(a,b,c) \,:\, (a,b) \in E',\, c \in C' \,\} \;.
	\end{align*}
	Hence in particular
	\begin{equation}\label{red-tQ}
	t(Q)=|E'||C'| \;.
	\end{equation}
	Observe that, by definition, $(a,b,c) \in E(H) \cap T(Q)$ if and only if $((a,b),c) \in E_G(E',C')$.
	Together with~(\ref{red-tQ}) this implies that, indeed,
	$$d_{G}(E',C') = \frac{e_{G}(E',C')}{|E'||C'|} = \frac{|E(H) \cap T(Q)|}{t(Q)} = d_{H}(Q) \;.$$
	Similarly, $T(P) = \{(a,b,c) \,:\, (a,b) \in E,\, c \in C\}$, hence 
	\begin{equation}\label{red-tP}
	t(P)=|E||C| \;.
	\end{equation}
	Therefore,
	$(a,b,c) \in E(H) \cap T(P)$ if and only if $((a,b),c) \in E_G(E,C)$,
	which together with~(\ref{red-tP}) implies that
	$$d_{G}(E,C) = \frac{e_{G}(E,C)}{|E||C|} = \frac{|E(H) \cap T(P)|}{t(P)} = d_{H}(P) \;.$$
	
	Having completed the proof of~(\ref{eq:red-d}),
	it now remains to prove~(\ref{eq:red-goal}).
	Let $\{P_i\}_i$ denote the set of triads of $(\Z,\E)$ that are subtriads of $P$ of the form $(E,E_1,E_2)$ with
	$E_1 \in \E_1$ and $E_2 \in \E_2$.
	Since $\E_1$ and $\E_2$ are partitions of $B \times C$ and $A \times C$, respectively, we have that
	\begin{equation}\label{eq:red-partitionP}
	T(P) = \bigcup_i T(P_i) \,\text{ is a partition of $T(P)$.}
	\end{equation}
	Furthermore, for each $P_i$ let $Q_i$ denote its subtriad induced by $E',C'$. Thus, each $Q_i$ is of the form $(E',E_1[A,C'],E_2[A,C'])$.
	It follows from~(\ref{eq:red-partitionP}) that
	\begin{equation}\label{eq:red-partitionQ}
	T(Q) = \bigcup_i T(Q_i) \,\text{ is a partition.}
	\end{equation}
	As $(\Z,\E)$ is an $(\ell,t,\e_2)$-equipartition,
	all graphs in $\E$ are $\e_2$-regular of density
	$1/\ell \pm \e_2$.
	Thus, applying Fact~\ref{lemma:triangle-counting} on $P_i$ implies that
	\begin{align}
	\begin{split}\label{eq:red-tP}
	t(P_i) &= \big( d(E)(1/\ell \pm \e_2)^2 \pm 7\e_2 \big)|A||B||C| \\
	&= \big( d(E)/\ell^2 \pm 10\e_2 \big) |A||B||C|
	= \big( 1 \pm 20\ell^3\e_2 \big) |E||C|/\ell^2
	= \big( 1 \pm \g \big) \cdot t(P)/\ell^2\;,
	\end{split}
	\end{align}
	with $\g := \d'/8 = \sqrt{\d}/4$,
	where the third equality bounds the error term using $d(E) \ge 1/\ell-\e_2 \ge 1/2\ell$,
	and the fourth equality uses~(\ref{red-tP}) and the statement's assumption on $\e_2$.
	Since $|C'| \ge \d|C|$, Claim~\ref{claim:slice2} implies that $E_1[A,C']$ and $E_2[B,C']$ are each $2\e_2/\d'$-regular of density $1/\ell \pm 2\e_2$.
	Thus, applying Fact~\ref{lemma:triangle-counting} on $Q_i$ implies that
	\begin{align}
	\begin{split}\label{eq:red-tQ}
	t(Q_i) &= \big( d(E')(1/\ell \pm 2\e_2)^2 \pm 14\e_2/\d \big)|A||B||C'| \\
	&= \big( d(E')/\ell^2 \pm 22\e_2/\d \big) |A||B||C'|
	= \big( 1 \pm 22\ell^3\e_2/\d^{3/2} \big) |E'||C'|/\ell^2
	= \big( 1 \pm \g \big) \cdot t(Q)/\ell^2  \;,
	\end{split}
	\end{align}
	where the third equality bounds the error term using the assumption $|E'| \ge \d'|E|$, which implies $d(E') \ge \d' d(E) \ge \sqrt{\d}/\ell$,
	and the fourth equality uses~(\ref{red-tQ}) and the statement's assumption on $\e_2$.
	In particular, using the assumptions $|E'| \ge \d'|E|$, $|C'| \ge \d'|C|$ (and~(\ref{red-tQ}),~(\ref{red-tP})), we deduce the lower bound
	\begin{equation}\label{eq:red-threshold}
	\frac{t(Q_i)}{t(P_i)}
	\ge \frac{1-\g}{1+\g}(\d')^2
	\ge \frac34 \cdot (2\sqrt{\d})^2 \ge \d \;,
	\end{equation}
	where we used the inequality
	\begin{equation}\label{eq:red-quotient}
	\frac{1-\g}{1+\g} \ge 1-2\g = 1-\frac14\d' \ge \frac34 \;.
	\end{equation}
	
	Recall that our goal is to prove~(\ref{eq:red-goal}).
	For a triad $X$ we henceforth abbreviate $e_H(X):=|E(H) \cap T(X)|=t(X) \cdot d_H(X)$.
	By~(\ref{eq:red-partitionP}) we have
	\begin{align*}
	e_H(P) &= \sum_i e_H(P_i)
	= \sum_i d_H(P_i) \cdot t(P_i)
	\le (1+\g)t(P) \cdot \frac{1}{\ell^2}\sum_i d_H(P_i) \;,
	\end{align*}
	where the inequality uses~(\ref{eq:red-tP}).
	Let $d=(1/\ell^2)\sum_i d_H(P_i)$. Then
	\begin{equation}\label{eq:red-dP}
	d_H(P) \le (1 + \g)d \;.
	\end{equation}
	The statement's assumption on the regularity of $(\Z,\E)$
	implies, together with~(\ref{eq:red-threshold}), that
	\begin{equation}\label{eq:red-reg}
	d_H(Q_i) \ge \frac23 d_H(P_i) \;.
	\end{equation}
	By~(\ref{eq:red-partitionQ}) we have
	\begin{align*}
	e_H(Q) &= \sum_i e_H(Q_i) =
	\sum_i d_H(Q_i) \cdot t(Q_i)
	\ge \sum_i \frac23 d_H(P_i) \cdot t(Q_i)
	\ge \frac23(1-\g)t(Q)\frac{1}{\ell^2}\sum_i d_H(P_i) \;,
	\end{align*}
	where the first inequality uses~(\ref{eq:red-reg}) and the second inequality uses~(\ref{eq:red-tQ}).
	This means that
	$$d_H(Q) \ge \frac23(1-\g)d \ge \frac23\cdot\frac{1-\g}{1+\g}d_H(P) \ge \frac12 d_H(P) \;,$$
	where the second inequality uses~(\ref{eq:red-dP}) and the third inequality uses~(\ref{eq:red-quotient}).
	We have thus proved~(\ref{eq:red-goal}) and are therefore done.	
\end{proof}

\end{document}